\documentclass{amsart}
\usepackage{amsmath,amsfonts,amsthm}
\usepackage{indentfirst}
\usepackage{hyperref}
\usepackage{graphicx,color}
\usepackage[all,cmtip]{xy}
\allowdisplaybreaks[4]
\numberwithin{equation}{section}
\newtheorem{lemma}{Lemma}[section]

\newtheorem{proposition}[lemma]{Proposition}
\newtheorem{theorem}[lemma]{Theorem}

\numberwithin{equation}{section}

\begin{document}
\title[Geodesic Connecting Lagrangian Graphs]
{On the existence of geodesic connecting Lagrangian graphs in $\mathbb{C}^n$}
\author{Yiyan Xu}
\date{\today}
\address{Yiyan Xu, The Department of Mathematics, Nanjing
University, Nanjing, China, 210093}
\email{xuyiyan@nju.edu.cn}
%\author{Pengfei Guan and Yiyan Xu}
%\address{Department of Mathematics and Statistics,
%McGill University,
%Montreal, Quebec H3A 0B9,  Canada}
%\email{guan@math.mcgill.ca}
%
%\address{Department of Mathematics, Nanjing
%University, Nanjing, China, 210093 and Department of Mathematics and Statistics,
%McGill University,
%Montreal, Quebec H3A 0B9,  Canada}
%	 \email{xuyiyan@math.pku.edu.cn}
%\thanks{Research of the first author was supported in part by NSERC Discovery Grant.}

\begin{abstract}
In this paper we show that two Lagrangian graphs over the torus in $\mathbb{C}^n$ with large Lagrangian phase can be connected via Lipschitz continuous geodesic with respect to the $L^2$ metric on the space of Lagrangian submanifolds. In particular, the geodesic for Lagrangian graphs over the torus in $\mathbb{C}^n$ can be formulated as a degenerate elliptic equation, and we construct geodesic by solving the corresponding Dirichlet problem.
\end{abstract}
%\subjclass{53C20, 53C21, 35J60}
\maketitle

\section{Introduction}
One of the important open problems regarding the geometry of Calabi-Yau manifolds consists in determining when a given Lagrangian admits a minimal Lagrangian (special Lagrangian submanifold) in its homology class or Hamiltonian class. The stability of an exact isotopy class should be related to the existence of a special Lagrangian representative, see \cite[Conjecture 5.2]{Thomas01} and \cite[Conjecture 7.3]{TY02}. Analogue to the study  Einstein-Hermitian metrics on holomorphic vector bundles and K\"{a}hler Einstein metrics on Fano manifolds, Solomon designed a possible program concerning the existence of special Lagrangian submanifolds \cite{Sol13,Sol14}.

To carry out the program, as noted in \cite{Sol14}, it is desirable to develop a satisfactory existence theory for geodesics on the space of Lagrangian submanifolds.
For the Hamiltonian isotopy class of $O(n)$-invariant Lagrangian sphere in Milnor fiber, Solomon and Yuval \cite{SY15}  constructed geodesics by using $O(n)$ symmetry to reduce the problem to ODE. In a resent work,
Rubinstein and Solomon \cite{RS15} studied the existence problem of geodesics for positive Lagrangian graphs over bounded domain in $\mathbb{C}^n$, where they used and extended the  Dirichlet duality theory for elliptic operator that developed by Harvy-Lawson \cite{HL09}. In this paper, we also concern the existence of geodesics for Lagrangian graphs. In particular, we construct Lipschitz continuous geodesics for Lagrangian graphs over the torus in $\mathbb{C}^n$ via degenerate elliptic PDE technique.

%Note that $\mathbb{T}^n\times \mathbb{T}^n$ can be identified with the quotient space of $\mathbb{C}^n$, %and inherits a K\"{a}hler structure $(J,\omega)$ and also nowhere vanishing holomorphicn $n$-form $\Omega$. %For example, if $z_1,\cdots, z_n$ are local coordinates of the torus, then %$\Omega=e^{-\sqrt{-1}\theta}dz_1\wedge \cdots \wedge dz_n$ is well defined nowhere vanishing holomorphic $n$-form.

%Let $\Lambda$ is a graph in $\mathbb{T}^n\times \mathbb{T}^n$ over the first factor, i.e. $\Lambda$ is the %image of some embedding $F: \mathbb{T}^n\rightarrow \mathbb{T}^n\times \mathbb{T}^n$, %$F(x)=x+\sqrt{-1}f(x)$. Then $\Lambda$ is Lagrangian if and only if the $1$-form $f_idx_i$ is closed. %Consequently, there exists (in general only locally
%defined) function $u$ such that $f=du$. In this case, $u$ is called the potential function of $\Lambda$.

Let $\Lambda$ be a graph in $\mathbb{C}^n$ over $\mathbb{T}^n=\mathbb{R}^n/\mathbb{Z}^n$, i.e. $\Lambda$ is the image of some embedding $F: \mathbb{T}^n\rightarrow \mathbb{C}^n$, $F(x)=x+\sqrt{-1}f(x)$; here $f$ can be regarded as a periodic function that is defined over the real factor $\mathbb{R}^n$, thus we may say the lagrangian graph over $\mathbb{T}^n$ is embedded in $\mathbb{C}^n$. Note that $\Lambda$ is Lagrangian if and only if the $1$-form $f_idx_i$ is closed. Consequently, there exists (in general only locally
defined) function $u$ such that $f(x)=\nabla u(x)$, see \cite{HL82}. In this case, $u$ is called the potential function of $\Lambda$.

An exact path of  lagrangian graph $\Lambda_t=(x,\nabla u(t,x))$ is a geodesic in the space of Lagrangians with respect to the $L^2$ metric that defined with the  holomorphic $n$-form $\Omega=e^{-\sqrt{-1}\theta}dz_1\wedge \cdots \wedge dz_n$, $\theta\in (-\pi, \pi]$, if the potential function $u$ satisfies the Lagrangian geodesic equation, see Proposition \ref{Lga-Geo_Equ_Pro_1},
\begin{equation}\label{LagGeoEqu_1}
{\rm Im}\Big(e^{-\sqrt{-1}\theta}\det\begin{bmatrix}
 \sqrt{-1}\frac{\partial^2u}{\partial t^2}&\sqrt{-1}\nabla\frac{\partial u}{\partial t}\\ (\sqrt{-1}\nabla\frac{\partial u}{\partial t})^T&I+\sqrt{-1}\nabla^2u
\end{bmatrix}\Big)=0.
\end{equation}

Suppose $\Lambda_i=(x,\nabla u_i(x))$ , $i=0, 1$, are two exact Lagrangian graph, then the existence of Lagrangian geodesic (graph) connecting $\Lambda_0$ and $\Lambda_1$ is equivalent to solve \eqref{LagGeoEqu_1} with boundary data
\[
\nabla u|_{t=0}=\nabla u_0,~~~\nabla u|_{t=1}=\nabla u_1.\]

 To construct solution of  equation \eqref{LagGeoEqu_1}, we try to approximate  the equation by a family of elliptic equation
\begin{equation}\label{LagGeoEquApp_1}
{\rm Im}\Big(e^{-\sqrt{-1}\theta}\det\begin{bmatrix}
 \tau+\sqrt{-1}\frac{\partial^2u}{\partial t^2}&\sqrt{-1}\nabla\frac{\partial u}{\partial t}\\ (\sqrt{-1}\nabla\frac{\partial u}{\partial t})^T&I+\sqrt{-1}\nabla^2u
\end{bmatrix}\Big)=0.
\end{equation}
As $\tau\rightarrow 0$, we recover the Lagranigan geodesic equation \eqref{LagGeoEqu_1} from \eqref{LagGeoEquApp_1}.
Now if we  scale the time variable, namely, by introducing the new variable
\[s=\sqrt{\tau}t\in [0,\sqrt{\tau}],\]
then
%\eqref{LagGeoEquApp_1} can be rewritten as
%\begin{equation}\label{LagGeoEquApp_3}
%{\rm Im}\Big(e^{-\sqrt{-1}\theta}\det\Big(\begin{bmatrix}
% \sqrt{\tau}&0\\
%  0&I
%\end{bmatrix}\begin{bmatrix}
%  1+\sqrt{-1}\frac{\partial^2}{\partial s^2}u &\sqrt{-1} \nabla\frac{\partial}{\partial s}u\\ (\sqrt{-1} %\nabla\frac{\partial}{\partial s}u)^T&I+\sqrt{-1}\nabla^2u
%\end{bmatrix}
%\begin{bmatrix}
% \sqrt{\tau}&0\\
%  0&I
%\end{bmatrix}\Big)\Big)=0.
%\end{equation}
%For $\tau>0$, then
\eqref{LagGeoEquApp_1} is equivalent to
\begin{equation}\label{SpeLagGeoEquApp_1}
{\rm Im}\big(e^{-\sqrt{-1}\theta}\det(I+\sqrt{-1}D^2u)\big)=0,
\end{equation}
or
\begin{equation}\label{SpeLagGeoEquApp_2}
\sum_{i=1}^{n+1}\arctan\lambda_i(D^2u)=\Theta=k\pi+\theta,~~~ {\rm for~some }~k\in \mathbb{Z},
\end{equation}
here we denote $D=(\frac{\partial}{\partial s}, \nabla)$, and $\lambda_i(D^2 u)$ to be the $i-$th eigenvalue of $D^2 u$.
In particular, the parameter $\tau$ disappear, and \eqref{SpeLagGeoEquApp_1} and \eqref{SpeLagGeoEquApp_2} are called the special Lagrangian equation \cite{HL82}. Harvey-Lawson\cite{HL82} showed that the special Lagrangian equation \eqref{SpeLagGeoEquApp_1} is elliptic at every solution.  Consequently, the Lagrangian geodesic equation \eqref{LagGeoEqu_1} is degenerate elliptic, see also \cite{RS15}.
The Dirichlet problem for equation \eqref{SpeLagGeoEquApp_1} on bounded domain $\Omega\subset \mathbb{R}^{n+1}$ was treated in \cite{CNS85} for the case where $D^2 u$ is required to lie on one of the two outermost branches, because of the concavity requirement. It is equivalent to require $\Theta\in[\frac{n-1}{2}\pi,\frac{n+1}{2}\pi)$ in \eqref{SpeLagGeoEquApp_2}, see \cite{CNS85, Yuan06}.

Let $\Lambda_i=(x,\nabla u_i(x))$ , $i=0, 1$, be two Lagrangian graph over $\mathbb{T}^n$, which satisfy $[du_0]=[du_1]\in H^1(\mathbb{T}^n)$. In particular, we can assume that $u_1-u_0$ is a globally defined function on $\mathbb{T}^n$. Denote
\begin{equation}
\vartheta(\Lambda_i)=\vartheta(\nabla^2u_i)=\sum_{i=1}^{n}\arctan\lambda_i(\nabla^2u_i),
\end{equation}
to be the Lagrangian phase for $\Lambda_i$ , $i=0, 1$, see \eqref{Lag-phase-fun-def}. Moreover, we set
\begin{equation}\label{Lin-Lag-Pot-1}
\tilde{u}^\tau(s,x)=(1-\frac{s}{\sqrt{\tau}})u_0+\frac{s}{\sqrt{\tau}}u_1=u_0+\frac{s}{\sqrt{\tau}}(u_1-u_0),
\end{equation}
 and
\begin{equation}\label{Hes-Lin-Lag-1}
\begin{split}
\chi^\tau&= D^2\tilde{u}^\tau=\begin{bmatrix}
 0 &\frac{1}{\sqrt{\tau}}\nabla (u_1-u_0)\\
  \frac{1}{\sqrt{\tau}}\nabla (u_1-u_0)^T&(1-\frac{s}{\sqrt{\tau}})\nabla^2u_0+\frac{s}{\sqrt{\tau}}\nabla^2u_1
\end{bmatrix},
\end{split}
\end{equation}
for convenience.
For each $0< \tau \leq 1$, we first solve the Dirichlet problem for the special Lagrangian equation \eqref{SpeLagGeoEquApp_2} over the cylinder $[0,\sqrt{\tau}]\times \mathbb{T}^n$ via the continuity method.

\begin{theorem}\label{SLE-Cyl-Exi-Thm}
For any fixed parameter $0< \tau \leq 1$, let $\chi^\tau$ be defined in \eqref{Hes-Lin-Lag-1}. Given constant $\Theta\in[\frac{n-1}{2}\pi,\frac{n+1}{2}\pi)$, assume
\begin{equation}\label{Adm-Bou-Con-1}
\vartheta(\nabla^2u_i)>\Theta-\frac{\pi}{2},~~i=0,1,
\end{equation}
%\[\delta=\frac{1}{2}\min\{\min_{x\in M}\vartheta_0(x)+\frac{\pi}{2}-\Theta, \min_{x\in M}\vartheta_1(x)+\frac{\pi}{2}-\Theta\}>0\]
there exists a smooth solution $v^\tau(s,x): [0,\sqrt{\tau}]\times \mathbb{T}^n\rightarrow \mathbb{R}$ to the Dirichlet problem of special Lagrangian equation
\begin{equation}\label{SLE1}
f(\lambda(\chi^\tau+D^2v^\tau))=\sum_{i=1}^{n+1}\arctan\lambda_i(\chi^\tau+D^2v^\tau)=\Theta,~~~\hbox{in}~~ [0,\sqrt{\tau}]\times \mathbb{T}^n
\end{equation}
with zero boundary data
\begin{equation}\label{SpeLagDirProTauBou1}
v^\tau|_{s=0}=0,~~~v^\tau|_{s=\sqrt{\tau}}=0.
\end{equation}
\end{theorem}
Moreover, for the sequence of solution \[\hat{v}^\tau(t,x):=v^\tau(\sqrt{\tau}t,x),\quad t\in [0,1],\] for \eqref{SLE1},  we have the uniform estimate,
\begin{equation}
|\hat{v}^\tau(t,x)|_{C^{1}([0,1]\times \mathbb{T}^n)}\leq C,
\end{equation}
where $C$  depends only on $u_0$ and $u_1$, not on the parameter $\tau$.   Let $\tau$ go to zero, with the continuity of the operator
\begin{equation}\label{SpeLagGeoOpe1}
{\rm Im}\Big(e^{-\sqrt{-1}\theta}\det\begin{bmatrix}
 \tau+\sqrt{-1}\frac{\partial^2u}{\partial t^2}&\sqrt{-1}\nabla\frac{\partial u}{\partial t}\\ (\sqrt{-1}\nabla\frac{\partial u}{\partial t})^T&I+\sqrt{-1}\nabla^2u
\end{bmatrix}\Big)
\end{equation}
in the topology of uniform convergence for convex functions\cite{TW97}, the limit \[u(t,x):=\tilde{u}^1(t,x)+\lim_{\tau\rightarrow 0}\hat{v}^\tau(t,x)\] will be the potential for the Lagrangian geodesic path.
\begin{theorem}\label{LGE-Exi-Thm}
Let $\Lambda_i=(x,\nabla u_i(x))$ , $i=0, 1$, be two Lagrangian graph over $\mathbb{T}^n$, which satisfy $[du_0]=[du_1]\in H^1(\mathbb{T}^n)$. Assume the Lagrangian phase satisfies $\vartheta(\Lambda_i)>\frac{n-1}{2}\pi$, then  $\Lambda_0$ and $\Lambda_1$ can be connected by a weak geodesic $\Lambda_t=(x,\nabla u(t,x))$
on the space of positive Lagrangian submanifolds with respect to  the  holomorphic $n$-form $\Omega=e^{-\sqrt{-1}\theta}dz_1\wedge \cdots \wedge dz_n$ for some $\theta\in (-\pi, \pi]$, see \eqref{Pre-Lag-Phase-1}. Moreover, the potential function is Lipschitz continuous solution of \eqref{LagGeoEqu_1}, i.e. $u(t,x)\in C^{0,1}([0,1]\times \mathbb{T}^n)$. The same result also holds for $\vartheta(\Lambda_i)<-\frac{n-1}{2}\pi, i=0, 1$.

\end{theorem}

\noindent\textbf{Acknowledgement}: Part of this work was done when the author was supported by CRC Postdoctoral Fellowship at McGill University between 2013-2014. The author would like to thank Professor Pengfei Guan, Professor Gang Tian and also my colleague  Yalong Shi for their useful discussions.

\section{The space of Lagrangians}

\subsection{Differential structure on the space of Lagrangians}
Let $(X,\omega)$ be a $2n$-dimensional symplectic manifold. Let $L$ be a (possibly non-compact) connected $n$-dimensional submanifold of $(X,\omega)$. We denote
\[\widehat{\mathcal{L}}(L,X)=\Big\{\iota\in {\rm Emb}(L,X)\Big|\iota^*\omega=0\Big\}\]
the space of Lagrangian embeddings of $L$ into $X$. If $L$ is non-compact, we impose that all $\iota\in \widehat{\mathcal{L}}(L,X)$ agree with a given $\iota_0\in \widehat{\mathcal{L}}(L,X)$ outside a compact subset. The group of compactly supported diffeomorphisms ${\rm Diff}(L)$ of $L$ acts on $\widehat{\mathcal{L}}(L,X)$ by $\iota\mapsto \iota\circ \phi$ for $\phi \in {\rm Diff}(L)$. Two Lagrangian embeddings $\iota_1, \iota_2\in {\rm Emb}(L,X)$ belongs to the same  ${\rm Diff}(L)$-orbit if and only if they have the same image $\Lambda=\iota_1(L)=\iota_2(L)$. Therefore, the quotient space $\mathcal{L}(L,X):=\widehat{\mathcal{L}}(L,X)/{\rm Diff}(L)$ can be considered as the space of Lagrangian submanifolds of $X$ which are diffeomorphic to $L$.

A path $\{\Lambda_t\}\subset\mathcal{L}(L,X)$ is said to be smooth if there exists a smooth map $[0,1]\times L\rightarrow X: (t,x)\mapsto \iota_t(x)$ such that $\iota_t(L)=\Lambda_t$ for all $t\in [0,1]$. This $\{\iota_t\}$ is called a lift of $\{\Lambda_t\}$.

Next we explain that one can think of $\mathcal{L}(L,X)$ as an infinite dimensional manifold \cite{AS01,IO07}. Let $[0,1]\times L\rightarrow X:(t,x)\mapsto \iota_t(x)$ be a smooth map such that $\iota_t\in \widehat{\mathcal{L}}(L,X)$ for all $t\in [0,1]$. Let us introduce a one-form on $L$ defined by
\begin{equation}\label{Tangent-Vector-PreLagrange-Space}
\alpha_t:=\omega\Big(\frac{d}{dt}\iota_t, d\iota_t\cdot\Big)\in \Omega^1(L),
\end{equation}
with Cartan's formula, we have
\[0=\frac{\partial}{\partial t}(\iota_t^*\omega)=d\alpha_t,\]
and hence the tangent space of $\widehat{\mathcal{L}}(L,X)$ at $\iota$ is given by
\[T_\iota\widehat{\mathcal{L}}(L,X)=\{v\in C^\infty(L,\iota^*TX)~|~\omega (v, d\iota\cdot)\in \Omega^1(L): {\rm closed}\}.\]
The tangent space to the ${\rm Diff}(L)$-orbit at $\iota$ is described as
\[T_\iota(\iota\cdot {\rm Diff}(L))=\{v=d\iota\circ \xi|\xi\in \mathfrak{X}(L)\},\]
where $\mathfrak{X}(L)$ denotes the space of vector fields on $L$. The map $v\mapsto \omega (v,d\iota\cdot)$ induces a linear map
\[T_\iota\widehat{\mathcal{L}}(L,X)/T_\iota(\iota\cdot {\rm Diff}(L))\rightarrow \{\beta\in \Omega^1(L)|d\beta=0\},\]
which is an isomorphism since $\iota: L\longrightarrow X$ is Lagranigan.

Let $[0,1]\longrightarrow \mathcal{L}(L,X): t\mapsto \Lambda_t$ be a smooth path of Lagrangian submanifolds. We define the velocity vector of the path $\{\Lambda_t\}$ at time  $t$ by
\[\frac{d}{dt}\Lambda_t:={\iota_t}_*\alpha_t,\]
where $\{\iota_t\}$ is a lift of $\{\Lambda_t\}$ and $\alpha_t$ is the one-from defined by \eqref{Tangent-Vector-PreLagrange-Space}. Moreover, $\beta_t:={\iota_t}_*\alpha_t \in \Omega^1(\Lambda_t)$ is closed and independent of the choice of the lift $\{\iota_t\}$.

\begin{lemma}
 The tangent space of $\mathcal{L}(L,X)$ at $\Lambda$ can be identified with the space of closed one forms (compact supported) on $\Lambda$, i.e.
\[T_{\Lambda}\mathcal{L}(L,X)=\{\beta\in \Omega^1(\Lambda)|d\beta =0\}.\]
\end{lemma}

% Recall that a family of symlectomorphisms $\phi_t$, $1\leq t\leq 1$, which starts at $\phi_0=id$ is called a symplectic isotopy of $M$.  Any such isotopy is generated by a unique %family of vector fields $X_t: M\rightarrow TM$ such that
%\[\frac{d}{dt}\phi_t=X_t\circ \phi_t.\]
%Since $\psi_t$ is a symplectomorphism for every $t$, the vector fields $X_t$ are symplectic and so
%\[d\iota(X_t)\omega =0.\]
%If all these $1$-froms are exact then there exists a smooth family of Hamiltonian functions $H_t: M\rightarrow \mathbb{R}$ such that
%\[\iota(X_t)\omega= dH_t.\]
%In this case the family $H_t$ is called a time-dependent Hamiltonian and $\psi_t$ is called
%Hamiltonian isotopy. In particular, if $M$ is simply connected then every symplectic isotopy is Hamiltonian.  A symplectomorphism $\psi\in {\rm Symp}(M,\omega)$ is called %Hamiltonian if there exists a Hamiltonian isotopy $\psi_t\in {\rm Symp}(M,\omega)$ from $\psi_0=id$ to $\psi_1=\psi$. We denote the space of Hamiltonian symplectomorphism by ${\rm %Ham}(M,\omega)$. It turns out that ${\rm Ham}(M,\omega)$ is a normal subgroup if ${\rm Symp}(M,\omega)$ and that its Lie algebra is the algebra of all Hamiltonian vector fields.

A path $\{\Lambda_t\}\subset\mathcal{L}(L,X)$  of Lagrangian submanifolds is called an exact Lagrangian path connecting $\Lambda_0$ and $\Lambda_1$, if there exists a compactly supported Hamiltonian isotopy $\{\psi_t\}_{0\leq t\leq 1}$ of $X$ such that $\psi_t(\Lambda_0)=\Lambda_t$ for every $t\in [0,1]$. From the work of Akveld-Salamon \cite{AS01}, a smooth path $\{\Lambda_t\}_{0\leq t\leq 1}$ in $\mathcal{L}(L,X)$ whose velocity vectors $\beta_t\in \Omega^1(\Lambda_t)$ are exact for all $t\in [0,1]$ is nothing but an exact path.

\begin{lemma}\label{Exact-Path-Exact-Tangent-vector}\cite{AS01}Let $\{\Lambda_t\}_{0\leq t\leq 1}\subset \mathcal{L}(L,X)$ be a smooth path of Lagrangian submanifolds. Then $\frac{d}{dt}\Lambda_t\in \Omega^1(\Lambda_t)$ is exact for every $t\in [0,1]$ if and only if $\{\Lambda_t\}$ is an exact path, that is, there exists a Hamiltonian isotopy $\{\psi_t\}$ such that $\psi_t(\Lambda_0)=\Lambda_t$ for every $t\in [0,1]$.
\end{lemma}

Lemma \ref{Exact-Path-Exact-Tangent-vector} says that a smooth path $\{\Lambda_t\}_{0\leq t\leq 1}\subset \mathcal{L}(L,X)$  is an exact path if and only if there exists a lift $\{\iota_t\}_{0\leq t\leq 1}$ of $\{\Lambda_t\}$ which satisfies
\begin{equation}\label{Exact-Tangent-vector-Hamilton-function}
 \alpha_t:=\omega (\frac{d}{dt}\iota_t, d\iota_t\cdot)=d(h_t\circ \iota_t)
\end{equation}
for functions $h_t\in C^\infty(\Lambda_t)$. Any $\{\psi_t\}$ satisfying $\psi_t(\Lambda_0)=\Lambda_t$ is generated by a Hamiltonian function $H\in C_0^\infty([0,1]\times X)$ such that $H_t|_{\Lambda_t}=h_t$ for each $t\in [0,1]$ and vice versa.

%Now we talk about isotopies of exact symplectic manifolds $(M, \omega=-d\lambda)$..
%\begin{proposition}\label{SymHamIsoExactPro}
%Suppose $\omega=-d\lambda$ and $\phi_t\in {\rm Diff}(M)$ is an isotopy starting at the identity %$\phi_0=id$. Then $\phi_t$ is a symplectic isotopy if and only if the $1-$form \[\phi_t^*\lambda-\lambda\] %is closed for every $t$. It is a Hamiltonian isotopy if and only if
%\[\phi_t^*\lambda-\lambda=dF_t\]
%for  a smooth family of functions $F_t: M\rightarrow \mathbb{R}$. In the case the function $F_t$ are %related to the Hamiltonian function $H_t$ by
%\begin{equation}\label{HamIsoExaFun1}
%F_t=\int_0^t(\iota(X_s)\lambda-H_s)\circ \phi_sds.
%\end{equation}
%\end{proposition}
%\begin{definition}
%A symplectomorphism $\phi$ of an exact symplectic manifold $(M,-d\lambda)$ is called \textbf{exact} with %respect to the form $\lambda$ if the the form $\phi^*\lambda-\lambda\in \Omega^1(M)$ is exact.
%
%A Lagrangian embedding $\iota: L\rightarrow M$ into an exact symplectic manifold $(M,-d\lambda)$ is called %\textbf{exact} with respect to the $1-$from $\lambda$ is $\iota^*\lambda$ is exact.
%\end{definition}
%
%\begin{corollary}
%Let $(M,\omega)$ be an exact symplectic manifold and $\phi_t$ be a symplectic isotopy. If $\phi_t$ is %Hamiltonian for each $t$ then $\phi_t$ is a Hamiltonian isotopy, i.e. it is generated by Hamiltonian vector %fields.
%\end{corollary}
%

\subsection{Lagrangians in a Calabi-Yau manifold}

Let $(X,J,\omega,\Omega)$ be a Calabi-Yau manifold, where $(X,J,\omega)$ is a $n$-dimensional K\"{a}hler manifold with complex structure $J$ and K\"{a}hler metric $\omega$, and $\Omega$ is a  nowhere vanishing holomorphicn $n$-form, with normalization $|\Omega|=1$.  For any point $p\in X$, there exist holomorphic coordinates $(z_1,\cdots, z_n)$ such that
\[\omega(p)=\frac{\sqrt{-1}}{2}\sum_{i=1}^ndz_i\wedge d\bar{z_i},~~~~\Omega(p)=e^{-\sqrt{-1}\theta}dz_1\wedge dz_2\wedge\cdots\wedge dz_n.\] For any oriented real $n$-plane $\tau\in T_pX$, Harvey-Lawson \cite[Prop 1.14]{HL82} showed the Lagrangian inequality hold,
\begin{equation}\label{HL-Lag-Ineq-1}
\big|{\rm Re}\Omega|_\tau\big|_{g}^2+\big|{\rm Im}\Omega|_\tau\big|_{g}^2\leq 1,
\end{equation}
with the equality if and only if $\tau$ is Lagrangian. For $\Lambda\subset X$ a Lagrangian submanifold, there exist $\vartheta:\Lambda\rightarrow \mathbb{R}/2\pi\mathbb{Z}$  such that
\begin{equation}\label{Phaasefun}
\Omega|_\Lambda=e^{\sqrt{-1}\vartheta}{\rm vol}_\Lambda,
\end{equation}
which is called  the Lagrangian  phase  function of $\Lambda$. The existence of Lagrangian  phase  function follows from the equality case of \eqref{HL-Lag-Ineq-1}.

 The Lagrange phase measures the rotation index of the angle of the tangent plane of the
Lagrangian submanifolds which illustrates an interesting interplay between symplectic
and Riemannian geometry of Lagrangian submanifolds, e.g. $d\vartheta=\iota_{\vec{H}}\omega$, here $\vec{H}$ is the mean curvature vector of $\Lambda$ \cite[(2.19)]{HL82}.

 Following Harvey-Lawson\cite{HL82}, $\Lambda$ is called special Lagrangian if the Lagrangian  phase  function of $\Lambda$ is constant. In the language of Calibrated geometry, special Lagrangian submanifold $\Lambda$ is calibrated by ${\rm Re}(e^{-\sqrt{-1}\theta}\Omega)$, or alternatively, ${\rm Im}(e^{-\sqrt{-1}\theta}\Omega)|_\Lambda=0$.

Now let consider a typical example. Suppose $\Lambda$ is a Lagrangian submanifold  of $\mathbb{C}^n$. Locally, $\Lambda$ can be described explicitly as the graph of  a function over a tangent plane. With no loss of generality, we may consider $\Lambda$ to be given as the graph over the axis plane $\mathbb{R}^n$, in $\mathbb{C}^n=\mathbb{R}^n+\sqrt{-1}\mathbb{R}^n$, of a function $y=f(x)$ where $z=x+\sqrt{-1}y$.
Moreover, $\Lambda$ is Lagrangian if and only if the $1$-form $f_idx_i$ is closed. Consequently, there exists (in general only locally
defined) function $u$ such that $f(x)=\nabla u(x)$, see \cite{HL82}. In this case, $u$ is called the potential function of $\Lambda$. For the graph $\Lambda=\{(x,\nabla u(x)),x\in \mathbb{R}^n\}\subset \mathbb{C}^n$, then
%  \[g=(\delta_{ij}+\frac{\partial^2u}{\partial x_i\partial x_k}\frac{\partial^2u}{\partial x_k\partial %x_j})dx_idx_j,\]
%  the induce volume is
%  \[{\rm vol}_\Lambda=\sqrt{\det g}=\sqrt{\det (Id+(\nabla^2u)^2)},\]
%  then
  \[\begin{split}
  e^{\sqrt{-1}\vartheta}&=\frac{\Omega|_\Lambda}{{\rm vol}_\Lambda}=\frac{e^{-\sqrt{-1}\theta}\det(I+\sqrt{-1}\nabla^2u)}{\sqrt{\det (Id+(\nabla^2u)^2)}}\\
  &=e^{-\sqrt{-1}\theta}\prod_{i=1}^n\frac{1+\sqrt{-1}\lambda_i}{\sqrt{1+\lambda_i^2}}\\
  &=e^{\sqrt{-1}(\sum\limits_{i=1}^n\arctan\lambda_i)-\theta},
  \end{split}\]
here $\lambda_i$ are the eigenvalue of $\nabla^2u$.
Then Lagrangian  phase can be given by
\begin{equation}\label{Lag-phase-fun-def}
\vartheta=\sum_{i=1}^n\arctan\lambda_i-\theta, ~~{\rm mod}~ 2\pi.
\end{equation}

\subsection{The space of positive Lagrangians in Calabi-Yau manifolds}
Following Solomon \cite{Sol13,Sol14}, a Lagrangian submanifold $\Lambda \in (X,J,\omega,\Omega)$ is positive if ${\rm Re}\Omega|_\Lambda$ is a volume form. Equivalently, the phase function of $\Lambda$ lie in the interval $(-\frac{\pi}{2}, \frac{\pi}{2})~{\rm mod}~2\pi$.
Denote by $\mathcal{L}^+=\mathcal{L}^+(X,L)\subset\mathcal{L}(X,L)$ the subspace of positive Lagrangian submanifolds. Let $\mathcal{O}\subset \mathcal{L}^+$ be a compactly supported exact isotopy class of positive Lagrangian submanifolds. That is, $\mathcal{O}$ is the collection of all $\Lambda\in \mathcal{L}^+$ that can be connected to a fixed point in $\mathcal{L}^+$ by an exact compact supported path. The isotopy class $\mathcal{O}$ is a submanifold of $\mathcal{L}^+$, and for $\Lambda\in \mathcal{O}$ the tangent space $T_\Lambda\mathcal{O}$ is canonically isomorphic to the space of exact $1$- forms on $\Lambda$ with compact support.  Let $\mathcal{H}_\Lambda$ denote the space of smooth function on $\Lambda$ with the following normalization: if $\Lambda$ is compact, then $\int_\Lambda h {\rm Re}\Omega=0$ and if $\Lambda$ is non-compact, then $h$ has compact support. With Lemma \ref{Exact-Path-Exact-Tangent-vector}, one can identify $T_\Lambda\mathcal{O}$ with $\mathcal{H}_\Lambda$.
Following \cite{Sol13,Sol14}, let
\[(\cdot,\cdot):H_\Lambda \times H_\Lambda \rightarrow \mathbb{R} \]
is given by
\begin{equation}\label{L2-metric-Pos-Lag-1}
(h,k)=\int_\Lambda hk{\rm Re}\Omega.
\end{equation}
Then $(\cdot,\cdot)$ define a Riemannian metric on $\mathcal{O}$. For a exact path of Lagrangian submanifolds $\{\Lambda_t\}_{0\leq t\leq 1}\subset \mathcal{O}\subset\mathcal{L}^+(L,X)$, i.e
\[\frac{d}{dt}\Lambda_t=dh_t,\] it is natural to define the energy of the exact path of Lagrangian submanifolds $ \{\Lambda_t\}$ by
\begin{equation}\label{Energy-Pos-Lag-1}
 E(\Lambda_t)=\int_0^1(h_t,h_t)dt=\int_0^1\int_{\Lambda_t}h_t^2{\rm Re}\Omega.
\end{equation}

An exact Lagrangian path $\{\Lambda_t\}$ is called a geodesic if $\{\Lambda_t\}$ is a critical point of the energy functional.

\section{Geodesic on the space Positive Lagrangian Submanifolds}

\subsection{Lagrangian Geodesic Equation}
Following Solomon's  definition of Levi-Civita connection and geodesics for the space of  Lagrangian submanifolds in Calabi-Yau manifold\cite{Sol13}, we deduced the geodesic equation  of Lagrangian graphs over $\mathbb{T}^n$ in $\mathbb{C}^n$. The geodesic equation  of Lagrangian graphs are also contained in \cite{RS15}.

For any $\alpha \in H^1(\mathbb{T}^n)$, we consider a family of positive Lagrangian graph,
\[\Lambda_t=(x, \nabla u(t,x))\in \mathcal{L}^+(\mathbb{C}^n,\mathbb{T}^n),\]
where $u(t,-): \mathbb{T}^n\rightarrow\mathbb{R}$ are locally defined function (up to a constant) and satisfies $[du(t,-)]=\alpha$ for any $t$. Furthermore, with  a normalization condition $\int_{\Lambda_t}\frac{\partial}{\partial t} u(t,x){\rm Re}\Omega=0$, one can choose a globally defined function $v(t,x): [a,b]\times \mathbb{T}^n\rightarrow \mathbb{R}$ such that $u(t,x)=u(0,x)+v(t,x)$. Then the tangent vector field along $\Lambda_t$ is

\[\frac{\partial}{\partial t}\Lambda_t=(0,\frac{\partial}{\partial t}\nabla u(t,x))=J\frac{\partial}{\partial t}\nabla u(t,x)=J\frac{\partial}{\partial t}\nabla v(t,x),\]
or \[\omega(\frac{\partial}{\partial t}\Lambda_t,-)=d(-\frac{\partial}{\partial t}u(t,x))=-\frac{\partial}{\partial t}v(t,x),\]
which is exact. From lemma \ref{Exact-Path-Exact-Tangent-vector}, we see that $\{\Lambda_t\}_{0\leq t\leq 1}$  is an exact Lagrangian path. Note that all the derivatives of $u$ are globally defined, thus we does not involve $v$ in the following expression for simplicity.

With \eqref{Energy-Pos-Lag-1}, the energy of $\Lambda_t$ is given by
\begin{eqnarray*}
% \nonumber to remove numbering (before each equation)
  E(\Lambda)&=&\frac{1}{2}\int_a^b(\frac{\partial u}{\partial t}, \frac{\partial u}{\partial t})dt \\
%  &=&\frac{1}{2}\int_a^b\int_{\Lambda_t}|\frac{\partial u}{\partial t}|^2{\rm %Re}\big(e^{-\sqrt{-1}\theta}\Omega\big) \\
  &=&\frac{1}{2}\int_a^b\int_{\mathbb{T}^n}|\frac{\partial u}{\partial t}|^2{\rm Re}\Omega.
\end{eqnarray*}
Consider a variation of Lagrangian path with fixed endpoints,
$\Lambda_{s,t}=(x, \nabla u(s,t,x)), (s,t)\in(-\epsilon,\epsilon)\times[a,b]$, with $\Lambda_{s,a}=\Lambda_a$, $\Lambda_{s,b}=\Lambda_b$. Then
\[\frac{\partial}{\partial t}\Lambda_{s,t}=(0,\frac{\partial}{\partial t}\nabla u(s,t,x))=d(-\frac{\partial}{\partial t}u(s,t,x)),\]
and
\[\frac{\partial}{\partial t}\Lambda_{s,t}\big|_{t=a, b}=0.\]
Now we compute the first variation of the energy functional,
\begin{eqnarray*}
% \nonumber to remove numbering (before each equation)
  E(\Lambda_{s,})&=&\frac{1}{2}\int_a^b\int_{\mathbb{T}^n}|\frac{\partial}{\partial t}u(s,t,x)|^2{\rm Re}\Omega,
\end{eqnarray*}
so,
\begin{equation}\label{Lag-Graph-Geodesic-1}
\begin{split}
  \frac{d}{ds}E(\Lambda_{s,}) =&\frac{1}{2}\frac{d}{ds}\int_a^b\int_{\mathbb{T}^n}|\frac{\partial u}{\partial t}|^2{\rm Re}\Omega\\
  =&\int_a^b\int_{\mathbb{T}^n} \frac{\partial^2u}{\partial s\partial t}\frac{\partial u}{\partial t}{\rm Re}\Omega+\frac{1}{2}\int_a^b\int_{\mathbb{T}^n}|\frac{\partial u}{\partial t}|^2\frac{\partial}{\partial s}{\rm Re}\Omega.
  \end{split}
\end{equation}
For the first term in \eqref{Lag-Graph-Geodesic-1}, one can change the differential order,
\begin{equation}\label{Lag-Graph-Geodesic-2}
\begin{split}
  \int_a^b&\int_{\mathbb{T}^n}\frac{\partial^2 u}{\partial s\partial t}\frac{\partial u}{\partial t}{\rm Re}\Omega\\
  =&\int_a^b\int_{\mathbb{T}^n} \frac{\partial}{\partial t}\big(\frac{\partial u}{\partial s}\frac{\partial u}{\partial t}{\rm Re}\Omega\big)-\frac{\partial u}{\partial s}\Big(\frac{\partial^2u}{\partial t^2}{\rm Re}\Omega+\frac{\partial u}{\partial t}\frac{\partial}{\partial t}{\rm Re}\Omega\Big).
  \end{split}
\end{equation}
For the second term in \eqref{Lag-Graph-Geodesic-1}, since $\Omega$ is of type $(n,0)$ and $\omega$ is of type $(1,1)$, we have
\[\omega\wedge{\rm Re}\Omega=0.\]
Let $\xi, \zeta$ be two Hamiltonian vector associated to the Hamiltonian function $H, K$, then we have
\[0=i_\xi i_\zeta(\omega\wedge{\rm Re}\Omega)=\{H,K\}{\rm Re}\Omega-dK\wedge i_\xi{\rm Re}\Omega+dH\wedge i_\zeta{\rm Re}\Omega+\omega\wedge i_\xi i_\zeta{\rm Re}\Omega.\]
Since $\Lambda$ is Lagrangian, then $\omega|_\Lambda=0$.
Moreover, by integration by parts, we have
\begin{equation}\label{Hamiltonian-Commutative-1}
\begin{split}
  \int_\Lambda\{H,K\}{\rm Re}\Omega&=\int_\Lambda -dH\wedge i_\zeta{\rm Re}\Omega+dK\wedge i_\xi{\rm Re}\Omega-\omega\wedge i_\xi i_\zeta{\rm Re}\Omega \\
  &=\int_\Lambda(Hdi_\zeta{\rm Re}\Omega-Kdi_\xi{\rm Re}\Omega).
  \end{split}
\end{equation}
Taking $H=\frac{1}{2}|-\frac{\partial u}{\partial t}|^2$, $K=-\frac{\partial u}{\partial s}$, $\xi=-\frac{\partial u}{\partial t}J\nabla\frac{\partial u}{\partial t}$, $\zeta=J\nabla\frac{\partial u}{\partial s}$ in \eqref{Hamiltonian-Commutative-1}, note that
\[\{H,K\}=\omega(\xi,\zeta)=-\frac{\partial u}{\partial t}\omega(J\nabla\frac{\partial u}{\partial t},J\nabla\frac{\partial u}{\partial s})=0,\]
thus
\begin{equation}\label{Lag-Graph-Geodesic-3}
\begin{split}
  \frac{1}{2}&\int_a^b\int_{\mathbb{T}^n}|-\frac{\partial u}{\partial t}|^2\frac{\partial}{\partial s}{\rm Re}\Omega\\
  =&\int_a^b\int_{\mathbb{T}^n}(-\frac{\partial u}{\partial s})d(-\frac{\partial u}{\partial t}i_{J\nabla\frac{\partial u}{\partial t}}{\rm Re}\Omega)\\
 =&\int_a^b\int_{\mathbb{T}^n}\frac{\partial u}{\partial s}d\frac{\partial u}{\partial t}\wedge i_{J\nabla\frac{\partial u}{\partial t}}{\rm Re}\Omega+\frac{\partial u}{\partial s}\frac{\partial u}{\partial t}\frac{\partial}{\partial t}{\rm Re}\Omega.
 \end{split}
\end{equation}
Combine the computation \eqref{Lag-Graph-Geodesic-2} and \eqref{Lag-Graph-Geodesic-3}, we have
\begin{eqnarray*}
  \frac{d}{ds}E(\Lambda_{s,}) &=&\int_a^b\int_{\mathbb{T}^n} \frac{\partial}{\partial t}\Big(\frac{\partial u}{\partial s}\frac{\partial u}{\partial t}{\rm Re}\Omega\Big)\\
  &&-\int_a^b\int_{\mathbb{T}^n}\frac{\partial u}{\partial s}\Big(\frac{\partial^2u}{\partial t^2}{\rm Re}\Omega-d\frac{\partial u}{\partial t}\wedge i_{-J\nabla\frac{\partial u}{\partial t}}{\rm Re}\Omega\Big)\\
&=&\int_{\mathbb{T}^n} \frac{\partial u}{\partial s}\frac{\partial u}{\partial t}{\rm Re}\Omega\Big|_a^b-\int_a^b\int_{\mathbb{T}^n}\frac{\partial u}{\partial s}(\frac{\partial^2 u}{\partial t^2}{\rm Re}\Omega-d\frac{\partial u}{\partial t}\wedge i_{J\nabla\frac{\partial u}{\partial t}}{\rm Re}\Omega).
\end{eqnarray*}

Consequently, a Lagrangian path $\Lambda_t=(x,\nabla u(t,x))$ is a critical point (geodesic) of the energy functional $E$  if and only if
 \begin{equation}\label{GeoEquLagGra1}
\frac{\partial^2 u}{\partial t^2}-\frac{d\frac{\partial u}{\partial t}\wedge i_{J\nabla\frac{\partial u}{\partial t}}{\rm Re}\Omega}{{\rm Re}\Omega}=0.
 \end{equation}

Now we will rewrite the equation \eqref{GeoEquLagGra1} in a explicit form. The holomorphic $n$-form
\[\Omega=e^{-\sqrt{-1}\theta}dz_1\wedge\cdots\wedge dz_n=e^{-\sqrt{-1}\theta}(dx_1+\sqrt{-1}dy_1)\wedge\cdots\wedge (dx_n+\sqrt{-1}dy_n),\] after restricted on the Lagrangian path, is reduced to be
\begin{eqnarray*}
% \nonumber to remove numbering (before each equation)
  \Omega|_{\Lambda_t}&=&e^{-\sqrt{-1}\theta}(dx_1+\sqrt{-1}d\frac{\partial u}{\partial x_1})\wedge\cdots\wedge (dx_n+\sqrt{-1}d\frac{\partial u}{\partial x_n})\\
  &=&e^{-\sqrt{-1}\theta}\det(I+\sqrt{-1}\nabla^2u)dx_1\wedge\cdots\wedge dx_n.
\end{eqnarray*}
Therefore, we have
\begin{eqnarray*}
% \nonumber to remove numbering (before each equation)
  &&d\frac{\partial u}{\partial t}\wedge i_{J\nabla\frac{\partial u }{\partial t}}{\rm Re}\Omega\big|_{\Lambda_t}\\
  &=&d\frac{\partial u}{\partial t}\wedge \sum_{i=1}^n(-1)^{i-1}{\rm Re}(\sqrt{-1}e^{-\sqrt{-1}\theta}\frac{\partial^2 u}{\partial t\partial x^i}\frac{\Omega}{dz^i})\big|_{\Lambda_t}\\
 % &=&\sum_{i=1}^n{\rm Re}(\sqrt{-1}|\frac{\partial^2}{\partial t\partial %x^i}u|^2\frac{1}{1+\sqrt{-1}\lambda_i}\det(I+\sqrt{-1}\nabla^2u)\\
%  &=&\sum_{i=1}^n{\rm Re}(|\frac{\partial^2}{\partial t\partial x^i}u|^2\frac{\sqrt{-1}}{1+\sqrt{-1}\lambda_i}\det(I+\sqrt{-1}\nabla^2u)\\
  &=&{\rm Re}(\sqrt{-1}e^{-\sqrt{-1}\theta}\nabla\frac{\partial u}{\partial t}\det (I+\sqrt{-1}\nabla^2u) (I+\sqrt{-1}\nabla^2u)^{-1} (\nabla\frac{\partial}{\partial t}u)^T).
    \end{eqnarray*}
Consequently,
%Denote $A=I+\sqrt{-1}\nabla^2u$, $a=\frac{\partial^2}{\partial t^2}u$, and %$X=\nabla\frac{\partial}{\partial t}u$,
%then
%\[ \begin{bmatrix}
%  a&X\\-X^T&A
%\end{bmatrix}\sim
%\begin{bmatrix}
%  a+XA^{-1}X^T&0\\0&A
%\end{bmatrix}=
%\begin{bmatrix}
%  a&0\\0&A
%\end{bmatrix}+\begin{bmatrix}
%  XA^{-1}X^T&0\\0&A
%\end{bmatrix}\]

 \begin{eqnarray*}
 &&\frac{\partial^2u}{\partial t^2}{\rm Re}\Omega-d\frac{\partial u}{\partial t}\wedge i_{J\nabla\frac{\partial u}{\partial t}}{\rm Re}\Omega\\
 &=&{\rm Re}\Big(e^{-\sqrt{-1}\theta}\det (I+\sqrt{-1}\nabla^2u)\big(\frac{\partial^2 u}{\partial t^2}-\sqrt{-1}\nabla\frac{\partial u}{\partial t} (I+\sqrt{-1}\nabla^2u)^{-1} (\nabla\frac{\partial u}{\partial t})^T\big)\Big)\\
 &=&{\rm Re}\Big(e^{-\sqrt{-1}\theta}\det\begin{bmatrix}
 \frac{\partial^2 u}{\partial t^2}-\sqrt{-1}\nabla\frac{\partial u}{\partial t}  (I+\sqrt{-1}\nabla^2u)^{-1} (\nabla\frac{\partial u}{\partial t})^T&0\\0&I+\sqrt{-1}\nabla^2u\end{bmatrix}\Big)\\
 %&=&{\rm Re}\det\begin{bmatrix}
 %-\sqrt{-1}(\sqrt{-1}\frac{\partial^2}{\partial t^2}u-\sqrt{-1}\nabla\frac{\partial}{\partial t}u  (I+\sqrt{-1}\nabla^2u)^{-1} (\sqrt{-1}\nabla\frac{\partial}{\partial %t}u)^T)&0\\0&I+\sqrt{-1}\nabla^2u\end{bmatrix}\\
 &=&{\rm Im}\Big(e^{-\sqrt{-1}\theta}\det\begin{bmatrix}
 \sqrt{-1}\frac{\partial^2 u}{\partial t^2}&\sqrt{-1}\nabla\frac{\partial u}{\partial t}\\ (\sqrt{-1}\nabla\frac{\partial u}{\partial t})^T&I+\sqrt{-1}\nabla^2u
\end{bmatrix}\Big).
\end{eqnarray*}
\begin{proposition}\label{Lga-Geo_Equ_Pro_1}
The exact Lagrangian path $\Lambda_t=(x,\nabla u(t,x))$ is a geodesic (critical point of the energy functional $E$)  if and only if
\begin{equation}\label{GeoEquEuc3}
{\rm Im}\Big(e^{-\sqrt{-1}\theta}\det\begin{bmatrix}
 \sqrt{-1}\frac{\partial^2}{\partial t^2}&\sqrt{-1}\nabla\frac{\partial u}{\partial t}\\ (\sqrt{-1}\nabla\frac{\partial u}{\partial t})^T&I+\sqrt{-1}\nabla^2u
\end{bmatrix}\Big)=0,
\end{equation}
or equivalently,
\begin{equation}\label{GeoEquEuc2}
\begin{split}
\cos\theta\sum_{k=0}^{\lfloor\frac{n}{2}\rfloor}&(-1)^k\sigma_{2k+1}(D^2_{t,x}u)-\sin\theta\sum_{k=0}^{\lfloor\frac{n+1}{2}\rfloor}(-1)^k\sigma_{2k}(D^2_{t,x}u)\\
&=\cos\theta\sum_{k=0}^{\lfloor\frac{n-1}{2}\rfloor}(-1)^k\sigma_{2k+1}(\nabla^2_{x}u)-\sin\theta\sum_{k=0}^{\lfloor\frac{n}{2}\rfloor}(-1)^k\sigma_{2k}(\nabla^2_{x}u),
\end{split}
\end{equation}
here $\sigma_k$ is the $k$-th elementary symmetric function, $k=0,1,2,\cdots$.
\end{proposition}

\section{The special Lagrangian equation with parameter $\tau$}
In this section, we will prove Theorem \ref{SLE-Cyl-Exi-Thm} via the continuity method and a priori estimates.
Firstly,  we recall a  beautiful linear algebra lemma from \cite[p272]{CNS85}-which we will use.
\begin{lemma}\label{AsyBehEigLem2}
\begin{enumerate}
  \item Consider the $(n+1)\times (n+1)$ symmetric matrix
\[A=\begin{bmatrix}
a&a_1&\cdots&a_{n}\\
a_1&\lambda_1'&\cdots&0\\
\vdots&\vdots&\ddots&\vdots\\
a_n&0&\cdots&\lambda'_{n}\\
\end{bmatrix}\]
with $\lambda_1', \cdots, \lambda_{n}'$ are fixed, $|a_i|\leq C$, $1\leq i\leq n$. If we let $|a|\rightarrow \infty$, then the eigenvalues of $A$  asymptotically behave like
\[\lambda_1=\lambda_1'+o(1),\cdots,\lambda_n=\lambda_{n}'+o(1), \lambda_{n+1}=a(1+O(\frac{1}{a})),\]
where $o(1)$ and $O(\frac{1}{a})$ are uniform-depending only on $\lambda_1',\cdots,\lambda_n'$ and $C$.

\item Similarly, let
\[A=\begin{bmatrix}
\frac{a}{\tau}&\frac{a_1}{\sqrt{\tau}}&\cdots&\frac{a_{n}}{\sqrt{\tau}}\\
\frac{a_1}{\sqrt{\tau}}&\lambda_1'&\cdots&0\\
\vdots&\vdots&\ddots&\vdots\\
\frac{a_{n}}{\sqrt{\tau}}&0&\cdots&\lambda_{n}'\\
\end{bmatrix}\]
be a $(n+1)\times (n+1)$ symmetric matrix depending a parameter $\tau\in(0,1]$, where $\lambda_1', \cdots, \lambda_{n}'$ are fixed, $|a_i|\leq C, 1\leq i\leq n$. If we let $|a|\rightarrow \infty$, then the eigenvalues of $A$  asymptotically behave like
\[\lambda_1=\lambda_1'+o(1),\cdots,\lambda_{n}=\lambda_{n-1}'+o(1), \lambda_{n+1}=\frac{a}{\tau}(1+O(\frac{1}{a})), \]
where $o(1)$ and $O(1)$ are uniform as $a\rightarrow \infty$, depending only on $\lambda_1'$, $\cdots$, $\lambda_{n}'$ and $C$, not on $\tau$.
\end{enumerate}
\end{lemma}
\begin{proof} The  part (1) of Lemma \ref{AsyBehEigLem2} is exactly contained in \cite{CNS85}; using the same argument, we  provide a proof of  part (2) for reader's convenience.  The eigenvalues $\lambda$ of $A$ satisfy
\[\det\begin{bmatrix}
1-\frac{\tau}{a}\lambda&\frac{a_1}{a}&\cdots&\frac{a_{n}}{a}\\
a_1&\lambda_1'-\lambda&\cdots&0\\
\vdots&\vdots&\ddots&\vdots\\
a_{n}&0&\cdots&\lambda_{n}'-\lambda\\

\end{bmatrix}=0.\]
Hence for $|a|=\infty$ the numbers $\lambda_1'$, $\cdots$, $\lambda_{n}'$ are roots. By continuity of the roots it follows that there are roots given by
\[\lambda_i=\lambda_i'+o(1),~~~i =1, \cdots, n.\]
To find the last eigenvalue, set $\lambda=a\mu$, then $\mu$ satisfies
\[\det\begin{bmatrix}
1-\mu\tau&\frac{a_1}{a}&\cdots&\frac{a_{n}}{a}\\
\frac{a_1}{a}&\frac{\lambda_1'}{a}-\mu&\cdots&0\\
\vdots&\vdots&\ddots&\vdots\\
\frac{a_{n}}{a}&0&\cdots&\frac{\lambda_{n}'}{a}-\mu\\
\end{bmatrix}=0.\]
For $|a|=\infty$, we see that $\mu=\frac{1}{\tau}$ is a simple root. By the implicit function theorem it follows that for $|a|$ large there is a root
\[\lambda_{n+1}=\frac{a}{\tau}(1+O(\frac{1}{a})).\]
\end{proof}

Now we will construct an admissible function $\underline{v}^\tau$ which is also a subsolution of \eqref{SLE1}.
Assuming that $u_0, u_1$, satisfies \eqref{Adm-Bou-Con-1}, we denote
\begin{equation}\label{Adm-Bou-Con-3}
\delta:=\frac{1}{2}\min\big\{\min_{x\in \mathbb{T}^n}\vartheta(\nabla^2u_0(x))+\frac{\pi}{2}-\Theta, \min_{x\in \mathbb{T}^n}\vartheta(\nabla^2u_1(x))+\frac{\pi}{2}-\Theta\big\}>0.
\end{equation}
For large parameter $\lambda$ consider
\begin{equation}\label{Adm-Sub-Sol-def}
\underline{v}^\tau=\frac{1}{2}\lambda \frac{s}{\sqrt{\tau}}(\frac{s}{\sqrt{\tau}}-1),
\end{equation}
then
\[\chi^\tau+D^2\underline{v}^\tau=\begin{bmatrix}
  \frac{\lambda}{\tau} &\frac{1}{\sqrt{\tau}}\nabla (u_1-u_0)\\
  \frac{1}{\sqrt{\tau}}\nabla (u_1-u_0)^T&\nabla^2\tilde{u}^\tau
\end{bmatrix},\]
here $\tilde{u}^\tau$ is defined in \eqref{Lin-Lag-Pot-1},  and $\chi^\tau$ is defined  in \eqref{Hes-Lin-Lag-1}.
By part (2) of lemma \ref{AsyBehEigLem2}, if $|\lambda|\rightarrow\infty$, we have
\begin{equation}\label{StrSubSol1}
\begin{split}
  f(\lambda(\chi^\tau+D^2\underline{v}^\tau))&=\sum_{i=1}^{n+1}\arctan\lambda_i(\chi^\tau+D^2\underline{v}^\tau)\\
   &=\arctan \Big(\frac{\lambda}{\tau}\big(1+O(\frac{1}{\lambda})\big)\Big)+\sum_{i=1}^n\big(\lambda_i(\nabla^2\tilde{u}^\tau)+o(1)\big)\\
  %&\geq  \arctan\frac{\lambda}{\tau}+(1-\frac{s}{\sqrt{\tau}})\vartheta_0+\frac{s}{\sqrt{\tau}}\vartheta_1\\
&\geq \arctan \frac{\lambda}{\tau}+O(\frac{1}{\lambda})+\Theta-\frac{\pi}{2}+2\delta+o(1),
%   &\geq\Theta+ \delta,
\end{split}
\end{equation}
where we used  the condition \eqref{Adm-Bou-Con-3} and the concave property of the eigenvalue as a function of matrix in the last inequality, and  $o(1)$ and $O(\frac{1}{\lambda})$ are uniform-only depending on $u_0$ and $u_1$, not on $\tau$. In particular, we can
take $\lambda\geq K_1=K_1(u_0,u_1)$ large, such that the error terms $O(\frac{1}{\lambda}), o(1)$ in \eqref{StrSubSol1} are dominated by $\delta$.
 Furthermore, if we take
 \begin{equation}\label{StrSubSol-Par-1}
 \lambda=\underline{\lambda}^\tau=\max\{K_1,\frac{\tau}{\tan  \delta }\},
 \end{equation}
 then
 \begin{equation}\label{StrSubSol3}
  f(\lambda(\chi^\tau+D^2\underline{v}^\tau))\geq\Theta+ \delta.
\end{equation}

\subsection{The continuity method:}
For any fixed $\tau\in (0,1]$, we follow the continuity method to solve the Dirichlet problem of special Lagragian equation \eqref{SLE1}.  Consider the set $E$ of all $\zeta\in [0,1]$ such that there exists a function $v=v^{\tau,\zeta}\in C^\infty([0,\sqrt{\tau}]\times \mathbb{T}^n)$ which solve the Dirichlet problem
\begin{equation}\label{ConSpeLagDirProTau}
f(\lambda(\chi^\tau+D^2v))=\sum_{i=1}^{n+1}\arctan\lambda_i(\chi+D^2v)=\varphi^\tau \quad{\rm in~} [0,\sqrt{\tau}]\times \mathbb{T}^n,
\end{equation}
with
\begin{equation}\label{ConSpeLagZeroBou}
v=0\quad {\rm on~} \{0,\sqrt{\tau}\}\times \mathbb{T}^n,
\end{equation}
and
\begin{equation}\label{ConSpeLagDirProConRig_1}
\varphi^\tau:=(1-\zeta)f(\lambda(\chi^\tau+D^2\underline{v}^\tau))+\zeta\Theta.
\end{equation}

Note that $v^{\tau,\zeta}$ is unique by the ellipticity and the comparison principle of special Lagrangian equation in the Hessian type, see \cite[Lemma B]{CNS85}. By the construction of $\varphi^\tau$ in \eqref{ConSpeLagDirProConRig_1}, for $\zeta=1$, the solution $v^\tau:=v^{\tau,1}$ of \eqref{ConSpeLagDirProTau} is the desired solution of the special Lagragian equation \eqref{SLE1}; for $\zeta=0$, the function $v^{\tau,0}:=\underline{v}^\tau$ is a solution of \eqref{ConSpeLagDirProTau}, i.e.  $0\in E$ and thus the set $E$ is non-empty.

Moreover, the linearized  operator of special Lagrangian equation \eqref{ConSpeLagDirProTau}
 \[L_{v}=F^{ij}(\chi+D^2v)D^2_{ij}=\frac{\partial}{\partial v_{ij}}f(\lambda(\chi+D^2v))D^2_{ij}\]is an strictly elliptic linear operator; it follows form the standard elliptic regularity and the inverse function theorem that $E$ is open.

In the following, we will assume $\zeta>\zeta_0=\zeta_0(\tau)>0$ and show $E$ is also closed and thus $\zeta=1\in E$. With the Schauder theory and standard elliptic bootstrapping argument, it is equivalent to derive the uniform a pori estimate
\[\|v^{\tau,\zeta}\|_{C^{2,\alpha}([0,\sqrt{\tau}]\times \mathbb{T}^n)}\leq C,\]
where $C$ does not depend on $\zeta$, but may depend on $\tau$. In the following, we  ignore the parameter $\tau$ for convenience.

\noindent\textbf{Subsolution:} By \eqref{StrSubSol3}, $\underline{v}=\underline{v}^\tau$ is an admissible subsolution to \eqref{ConSpeLagDirProTau}, and satisfies
\begin{equation*}
         f(\lambda(\chi+D^2\underline{v}))=\varphi
         +\zeta\big((\Theta(\chi^\tau+D^2\underline{v}^\tau)-\Theta\big)         \geq \varphi+\zeta\delta,~~~  {\rm in~}  [0,\sqrt{\tau}]\times \mathbb{T}^n.
       \end{equation*}
Since $v^\zeta=\underline{v}=0$ on $\{0,\sqrt{\tau}\}\times \mathbb{T}^n$,  by the comparison principle, we have
\begin{equation}\label{Upper-estimate-1}
v^\zeta\geq \underline{v} ~~~  {\rm in~} [0,\sqrt{\tau}]\times \mathbb{T}^n.
\end{equation}
Moreover, since $\Theta\geq\frac{n-1}{2}\pi$,  using the concavity we find that
\[f(\lambda(\chi+D^2\underline{v}))\leq f(\lambda(\chi+D^2v^\zeta))+F^{ij}(\chi+D^2v^\zeta)D_{ij}(\underline{v}-v^\zeta),\]
so that
\begin{equation}\label{SubSol-Barrier-1}
L_{v^\zeta}(\underline{v}-v^\zeta)\geq \zeta\delta.
\end{equation}

\noindent\textbf{Super function:}
Define
\[\bar{v}=\bar{v}^\tau=\frac{1}{2}\lambda \frac{s}{\sqrt{\tau}}(1-\frac{s}{\sqrt{\tau}}),\]
then
\[\chi+D^2\bar{v}=\begin{bmatrix}
  -\frac{\lambda}{\tau}&\frac{1}{\sqrt{\tau}}\nabla (u_1-u_0)\\
  \frac{1}{\sqrt{\tau}}\nabla (u_1-u_0)^T&\nabla^2\tilde{u}
\end{bmatrix},\]
and
\begin{equation*}
         {\rm tr}\chi+\Delta \bar{v} =-\frac{\lambda}{\tau}+\Delta \tilde{u},~~~  {\rm in~} [0,\sqrt{\tau}]\times \mathbb{T}^n.
\end{equation*}
Consequently, if we choose
\begin{equation}\label{Lower-estimate-Par-2}
\lambda=\bar{\lambda}^\tau =\max\{\sup_{\mathbb{T}^n}\Delta u_0,\sup_{\mathbb{T}^n}\Delta u_1\}\tau,
\end{equation}
then  we have
\[{\rm tr}\chi+\Delta \bar{v} \leq 0.\]
Since we have
\[{\rm tr}\chi+\Delta v^\zeta\geq0,\]
by the comparison principle and  $v^\zeta=\bar{v}=0$ on $\{0,\sqrt{\tau}\}\times \mathbb{T}^n$ again,
\begin{equation}\label{Lower-estimate-1}
v^\zeta\leq \bar{v} ~~~  {\rm in~} [0,\sqrt{\tau}]\times \mathbb{T}^n.
\end{equation}

\subsection{Preliminary a priori estimate}Now, combine with \eqref{Upper-estimate-1} and \eqref{Lower-estimate-1}, we have the $L_\infty$ estimate
\begin{equation}\label{L-infity-estimate-1}
-\frac{1}{8}\underline{\lambda}^\tau\leq \underline{v}\leq v^\zeta\leq \bar{v}\leq \frac{1}{8}\bar{\lambda}^\tau,
\end{equation}
here $\underline{\lambda}^\tau$ is defined in \eqref{StrSubSol-Par-1} and $\bar{\lambda}^\tau$ is defined in \eqref{Lower-estimate-Par-2}.

Consequently, for the interior normal derivative at $\{0, \sqrt{\tau}\}\times \mathbb{T}^n$,
\begin{equation}\label{Normal-derivative-estimate-low}
-\frac{1}{2\sqrt{\tau}}\underline{\lambda}^\tau\leq\frac{\partial}{\partial s}\Big|_{s=0}\underline{v}\leq\frac{\partial}{\partial s}\Big|_{s=0}v^\zeta\leq \frac{\partial}{\partial s} \bar{v}\Big|_{s=0}\leq \frac{1}{2\sqrt{\tau}}\bar{\lambda}^\tau,
\end{equation}
and
\begin{equation}\label{Normal-derivative-estimate-up}
\frac{1}{2\sqrt{\tau}}\underline{\lambda}^\tau\geq\frac{\partial}{\partial s}\Big|_{s=\sqrt{\tau}}\underline{v}\geq\frac{\partial}{\partial s}\Big|_{s=\sqrt{\tau}}v^\zeta\geq \frac{\partial}{\partial s} \bar{v}\Big|_{s=\sqrt{\tau}}\geq -\frac{1}{2\sqrt{\tau}}\bar{\lambda}^\tau.
\end{equation}
In particular,
\begin{equation}\label{Normal-derivative-estimate-1}
\sup_{\{0,1\}\times \mathbb{T}^n}|\frac{\partial}{\partial t}v^\tau|=\sup_{\{0,\sqrt{\tau}\}\times \mathbb{T}^n}|\sqrt{\tau}\frac{\partial}{\partial s}v^\tau|\leq\max\{\frac{1}{2}\underline{\lambda}^\tau,\frac{1}{2}\bar{\lambda}^\tau\}.
\end{equation}

Moreover, since $v^\zeta$ vanish identically on the boundary $\{0,\sqrt{\tau}\}\times \mathbb{T}^n$,  the tangential derivative also vanish, i.e.
\begin{equation}\label{Tangential-derivative-estimate-1}
\nabla v^\zeta\big|_{\{0,\sqrt{\tau}\}\times \mathbb{T}^n}=0.
\end{equation}

Next, we will establish the interior derivative estimate.
Differentiating equation \eqref{ConSpeLagDirProTau} in the direction of $\xi$, we obtain
\[L_{v^\zeta}D_{\xi}(\tilde{u}+v^\zeta)=D_{\xi}\varphi.\]
\noindent\textbf{ Case I}: If $\xi=\frac{\partial}{\partial s}$, then
\[\begin{split}
  D_{\xi}(\chi_{ij}+\underline{v}_{ij})&=\frac{\partial}{\partial s}\begin{bmatrix}
  -\frac{\underline{\lambda}^\tau}{\tau}&\frac{1}{\sqrt{\tau}}\nabla (u_1-u_0)\\
  \frac{1}{\sqrt{\tau}}\nabla (u_1-u_0)^T&(1-\frac{s}{\sqrt{\tau}})\nabla^2u_0+\frac{1}{\sqrt{\tau}}\nabla^2u_1
\end{bmatrix}\\
&=\begin{bmatrix}
  0&0\\
  0&\frac{1}{\sqrt{\tau}}(\nabla^2u_1-\nabla^2u_0)
\end{bmatrix};
\end{split}\]
\noindent\textbf{ Case II}: If $\xi=\frac{\partial}{\partial x^k}\in T\mathbb{T}^n$, then
\[\begin{split}
  D_{\xi}(\chi_{ij}+\underline{v}_{ij})&=\frac{\partial}{\partial x^k}\begin{bmatrix}
  -\frac{\underline{\lambda}^\tau}{\tau}&\frac{1}{\sqrt{\tau}}\nabla (u_1-u_0)\\
  \frac{1}{\sqrt{\tau}}\nabla (u_1-u_0)^T&(1-\frac{s}{\sqrt{\tau}})\nabla^2u_0+\frac{1}{\sqrt{\tau}}\nabla^2u_1
\end{bmatrix}\\
&=\begin{bmatrix}
  0&\frac{1}{\sqrt{\tau}}\nabla \frac{\partial}{\partial x^k}(u_1-u_0)\\
  \frac{1}{\sqrt{\tau}}\nabla \frac{\partial}{\partial x^k}(u_1-u_0)^T&(1-\frac{s}{\sqrt{\tau}})\nabla^2\frac{\partial}{\partial x^k}u_0+\frac{1}{\sqrt{\tau}}\nabla^2\frac{\partial}{\partial x^k}u_1
\end{bmatrix}.
\end{split}\]
In either case, from the definition of $\varphi$ in \eqref{ConSpeLagDirProConRig_1}, we have
\[\begin{split}
|D_{\xi}\varphi|&=(1-\zeta)|F^{ij}(\chi+D^2\underline{v})D_{\xi}(\chi_{ij}+\underline{v}_{ij})|\leq \frac{C_1(1-\zeta)}{\sqrt{\tau}},
\end{split}\]
here $C_1$  depends only on $u_0$, $u_1$.

If we take
\[A_1=\frac{C_1(1-\zeta)}{\zeta\delta\sqrt{\tau}},\]
so that, by \eqref{SubSol-Barrier-1}, we have
\[L_{v^\zeta}\big(A(\underline{v}-v^\zeta)\pm D_\xi (\tilde{u}+ v^\zeta)\big)\geq 0.\]
By using the maximum principle again, we conclude that the function
\[A(\underline{v}-v^\zeta)\pm D_\xi (\tilde{u}+ v^\zeta)\]
attains its maximum on the boundary, and thus
\begin{equation}|D_\xi (\tilde{u}+ v^\zeta)|\leq \sup_{[0,\sqrt{\tau}]\times \mathbb{T}^n}A_1(v^\zeta-\underline{v})+\sup_{\{0,\sqrt{\tau}\}\times \mathbb{T}^n}|D_\xi (\tilde{u}+ v^\zeta)|.
\end{equation}
Consequently, combine with \eqref{L-infity-estimate-1}, \eqref{Normal-derivative-estimate-low} and \eqref{Normal-derivative-estimate-up},
\begin{equation}\label{Interior-Tan-der-est-1}
\begin{split}
|\frac{\partial}{\partial s}v^\zeta|&\leq \sup_{[0,\sqrt{\tau}]\times \mathbb{T}^n}\Big(|\frac{\partial}{\partial s}\tilde{u}|+A_1(v^\zeta-\underline{v})\Big)+\sup_{\{0,\sqrt{\tau}\}\times \mathbb{T}^n}\Big(|\frac{\partial}{\partial s}\tilde{u}|+ |\frac{\partial}{\partial s}v^\zeta|\Big)\\
&=2\sup_{\mathbb{T}^n}\frac{|u_1-u_0|}{\sqrt{\tau}}+\frac{C_1(1-\zeta)}{\zeta\delta\sqrt{\tau}}\sup_{[0,\sqrt{\tau}]\times \mathbb{T}^n}(v^\zeta-\underline{v})+\sup_{\{0,\sqrt{\tau}\}\times \mathbb{T}^n}|\frac{\partial}{\partial s}v^\zeta|\\
&\leq \frac{C}{\sqrt{\tau}},
\end{split}
\end{equation}
here $C$ depends only on $u_0$, $u_1$ and $\zeta_0$; similarly,
\begin{equation}\label{Interior-Tan-der-est-2}
\begin{split}
|\nabla v^\zeta|&\leq \sup_{[0,\sqrt{\tau}]\times \mathbb{T}^n}\Big(|\nabla\tilde{u}|+A_1(v^\zeta-\underline{v})\Big)+\sup_{\{0,\sqrt{\tau}\}\times \mathbb{T}^n}\Big(|\nabla\tilde{u}|+ |\nabla v^\zeta|\Big)\\
&=2\max\{\sup_{\mathbb{T}^n}|\nabla u_0|,\sup_{\mathbb{T}^n}|\nabla u_1|\}+\frac{C_1(1-\zeta)}{\zeta\delta\sqrt{\tau}}\sup_{[0,\sqrt{\tau}]\times \mathbb{T}^n}(v^\zeta-\underline{v})\\
&\leq \frac{C}{\sqrt{\tau}},
\end{split}
\end{equation}
here $C$ depends only on $u_0$, $u_1$ and and $\zeta_0$. Furthermore, from \eqref{Interior-Tan-der-est-2}, for $\zeta=1$,
\begin{equation}\label{Interior-Tan-der-est-3}
\begin{split}
|\nabla v^{\tau,1}|\leq 2\max\{\sup_{\mathbb{T}^n}|\nabla u_0|,\sup_{\mathbb{T}^n}|\nabla u_1|\}.
\end{split}
\end{equation}
In summary, we have established the following first order derivative estimate:
\begin{equation}|\frac{\partial}{\partial s} v^{\tau,\zeta}|\leq \frac{C}{\sqrt{\tau}},~~~|\nabla v^{\tau,\zeta}|\leq \frac{C}{\sqrt{\tau}},\end{equation}
and also
\begin{equation}\label{Fir-Der-Est-t-1}
|\frac{\partial}{\partial t} v^{\tau,\zeta}|\leq C,~~|\nabla v^{\tau, 1}|\leq C,
\end{equation}
here $C$ depends only on $u_0$, $u_1$ and $\zeta_0$.

We next derive the second derivatives estimate. Since
\begin{equation}\label{Sec-Der-Low-Est-1}
{\rm tr}(\chi+D^2 v^\zeta)\geq 0,
\end{equation}
we only need to estimate the upper bound of the second derivatives.
Differentiating equation \eqref{ConSpeLagDirProTau} twice along the direction of $\xi$,
we have
\begin{equation}\label{Sec-Diff-Equ-1}
L_{v^\zeta}(\chi+D^2v^\zeta)(\xi,\xi)+F^{ij,kl}(\chi+D^2v^\zeta)D_{\xi}(\chi_{ij}+v^\zeta_{ij})D_{\xi}(\chi_{kl}+v^\zeta_{kl})
=D^2_{\xi\xi}\varphi.
\end{equation}
Similarly, as did in the first derivative estimate, we find that
\[D^2_{\xi\xi}(\chi+D^2\underline{v})=O(\frac{1}{\sqrt{\tau}}),\]
and thus
\begin{equation}\label{Sec-Diff-Phi-1}
|D^2_{\xi\xi}\varphi|\leq \frac{C_2(1-\zeta)}{\sqrt{\tau}},
\end{equation}
here $C_2$  depends only on $u_0$, $u_1$, not on $\tau$.

By the concavity again, from \eqref{Sec-Diff-Equ-1} and \eqref{Sec-Diff-Phi-1},
\[L_{v^\zeta}(\chi+D^2v^\zeta)(\xi,\xi)\geq -\frac{C_2(1-\zeta)}{\sqrt{\tau}}.\]
 If we take
\[A_2=\frac{C_2(1-\zeta)}{\zeta\delta\sqrt{\tau}},\]
So that, by \eqref{SubSol-Barrier-1},
\[L_{v^\zeta}\big(A_2(\underline{v}-v^\zeta)+(\chi+D^2v^\zeta)(\xi,\xi)\big)\geq 0.\]
Using the maximum principle again, we conclude that the function
\[A_2(\underline{v}-v^\zeta)+(\chi+D^2v^\zeta)(\xi,\xi)\]
attains its maximum on the boundary. Consequently,
\begin{equation}\label{Int-Sec-Der-Est-1}
(\chi+D^2v^\zeta)(\xi,\xi)\leq \sup_{[0,\sqrt{\tau}]\times \mathbb{T}^n}A_2(v^\zeta-\underline{v})+\sup_{\{0,\sqrt{\tau}\}\times \mathbb{T}^n}(\chi+D^2v^\zeta)(\xi,\xi).
\end{equation}

Now we have to estimate the second derivatives at the boundary point. Note that $v^\zeta$ vanish identically on the boundary $\{0,\sqrt{\tau}\}\times \mathbb{T}^n$,
therefore the double tangential derivative in the space direction are also vanish identically on the boundary, i.e.
\begin{equation}\label{Tan-Tan-Bou-1}
\nabla^2 v^\zeta\big|_{\{0,\sqrt{\tau}\}\times \mathbb{T}^n}=0.
\end{equation}

Next we estimate the mixed second derivative on the boundary.
By differential equation \eqref{ConSpeLagDirProTau} along the  space direction $\frac{\partial}{\partial x^k}\in T\mathbb{T}^n$, we obtain
\[L_{v^\zeta}\nabla_{\frac{\partial}{\partial x^k}}(\tilde{u}+v^\zeta)=\nabla_{\frac{\partial}{\partial x^k}}\varphi.\]
Thus we have
\begin{equation*}
|L_{v^\zeta}\nabla_{\frac{\partial}{\partial x^k}}v^\zeta|=|\nabla_{\frac{\partial}{\partial x^k}}\varphi-L_{v^\zeta}\nabla_{\frac{\partial}{\partial x^k}}\tilde{u}|\leq \frac{C_3}{\sqrt{\tau}},
\end{equation*}
where $C_3$   depends only on $u_0$, $u_1$, not on $\tau$.
If we take
\[A_3=\frac{C_3}{\zeta\delta\sqrt{\tau}},\]
by \eqref{SubSol-Barrier-1}, it follows that
\[L_{v^\zeta}\big(A_3(\underline{v}-v^\zeta)\pm \nabla_{\frac{\partial}{\partial x^k}} v^\zeta\big)\geq 0.\]
Moreover,
\[\big(A_3(\underline{v}-v^\zeta)\pm \nabla_{\frac{\partial}{\partial x^k}}v^\zeta\big)\Big|_{\{0,\sqrt{\tau}\}\times \mathbb{T}^n}=0,\]
the maximum principle implies that
\[A(\underline{v}-v^\zeta)\pm \nabla_{\frac{\partial}{\partial x^k}}v^\zeta\leq 0,~~~\hbox{in}~ [0,\sqrt{\tau}]\times \mathbb{T}^n.\]
Consequently,
\[\frac{\partial}{\partial s}\Big|_{s=0}A(\underline{v}-v^\zeta)\leq\frac{\partial}{\partial s}\Big|_{s=0}\nabla_{\frac{\partial}{\partial x^k}}v^\zeta\leq \frac{\partial}{\partial s}\Big|_{s=0} A(v^\zeta-\underline{v}),\]
\[\frac{\partial}{\partial s}\Big|_{s=\sqrt{\tau}}A(\underline{v}-v^\zeta)\geq\frac{\partial}{\partial s}\Big|_{s=\sqrt{\tau}}\nabla_{\frac{\partial}{\partial x^k}}v^\zeta\geq \frac{\partial}{\partial s}\Big|_{s=\sqrt{\tau}} A(v^\zeta-\underline{v}).\]
From \eqref{Interior-Tan-der-est-1} and the definition of $\underline{v}$ in \eqref{Adm-Sub-Sol-def}, we conclude that
\begin{equation}\label{Tan-Nor-Bou-1}
|\frac{\partial}{\partial s}\nabla v^\zeta|\leq \frac{C}{\tau},
\end{equation}
here $C$ depends on only on $u_0$, $u_1$ and and $\zeta_0$.

Finally, we will establish the a priori estimates of the upper bound of double normal derivative on the boundary. Let us choose $K_2=K_2(u_0,u_1,\tau)$ large, to determined later, if
\[\sup_{\{0,\sqrt{\tau}\}\times \mathbb{T}^n}\frac{\partial^2}{\partial s^2}v^\zeta\leq K_2,\]
then we are done; Otherwise, at some pint $p_0\in \{0,\sqrt{\tau}\}\times \mathbb{T}^n$,
\[\frac{\partial^2}{\partial s^2}v^\zeta(p_0)>K_2.\]
We may then apply the part (1) in lemma \ref{AsyBehEigLem2} and conclude that
\begin{equation}\label{BouSecDerEst-1}
  \varphi=f(\chi+D^2v^\zeta)\sim \arctan \frac{\partial^2}{\partial s^2}v^\zeta+\sum_{i=1}^n\arctan\lambda_i(\nabla^2\tilde{u})\quad {\rm at}~p_0.
\end{equation}
On the other hand, from \eqref{StrSubSol1} and \eqref{ConSpeLagDirProConRig_1}, we have
\begin{equation}\label{BouSecDerEst-2}
\begin{split}
\varphi&=(1-\zeta)\Theta(\chi+D^2\underline{v})+\zeta\Theta\\
&\sim(1-\zeta)\big(\arctan \frac{\underline{\lambda}^\tau}{\tau}+\sum_{i=1}^n\arctan\lambda_i(\nabla^2\tilde{u})\big)+\zeta\Theta.
\end{split}
\end{equation}
Combine with \eqref{BouSecDerEst-1} and \eqref{BouSecDerEst-2},
\[\begin{split}
\arctan \frac{\partial^2}{\partial s^2}v^\zeta(p_0)&\sim(1-\zeta)\arctan \frac{\underline{\lambda}^\tau}{\tau}+\zeta\big(\Theta-\sum_{i=1}^n\arctan\lambda_i(\nabla^2\tilde{u})\big)\\
&\leq (1-\zeta)\arctan \frac{\underline{\lambda}^\tau}{\tau}+\zeta (\frac{\pi}{2}-2\delta)\\
&\leq \max\{\arctan \frac{\underline{\lambda}^\tau}{\tau},\frac{\pi}{2}-2\delta\}.
\end{split}\]
We now take $K_2=K_2(u_0,u_1,\tau)$ large enough, such that
\[\arctan \frac{\partial^2}{\partial s^2}v^\zeta(p_0)\leq \max\{\arctan \frac{\underline{\lambda}^\tau}{\tau}+1,\frac{\pi}{2}-\delta\}.\]
Consequently,
\begin{equation}\label{Double-Nor-Bou-1}
  \frac{\partial^2}{\partial s^2}u^\zeta(p_0)\leq \max\{\frac{\underline{\lambda}^\tau}{\tau}+1,\frac{1}{\tan \delta}, K_2\}.
\end{equation}
From \eqref{Sec-Der-Low-Est-1}, \eqref{Int-Sec-Der-Est-1}, \eqref{Tan-Tan-Bou-1}, \eqref{Tan-Nor-Bou-1} and \eqref{Double-Nor-Bou-1}, the second derivatives are dominated by some constant $C$, under control.

In summery, we have
 \begin{equation}\label{ConUniC2Est1}
  |v^\zeta|_{C^2([0,\sqrt{\tau}]\times \mathbb{T}^n)}\leq C,
  \end{equation}
where $C$ depends on $u_0$, $u_1$, and  $\tau$,  but not on $\zeta$. Consequently, the equation \eqref{ConSpeLagDirProTau} is uniformly elliptic; moreover, $\varphi^\tau\geq \Theta\geq \frac{n-1}{2}\pi$, then the partial differential operator of the special Lagrangian equation is concave at the admissible function,  and the Evans-Krylov theorem implies the H\"{o}lder continuity of the second order derivatives. Then the standard elliptic bootstrapping argument using Schauder theory imples that the solution in fact smooth. The usual compactness argument shows that $E$ is closed and we conclude that $E=[0,1]$. In particular, $\zeta=1\in E$ as desired which proves Theorem \ref{SLE-Cyl-Exi-Thm}.

\section{Existence of weak Lagragian geodesic}
In this section, with Theorem \ref{SLE-Cyl-Exi-Thm}, we will prove Theorem \ref{LGE-Exi-Thm}.

Let $\Lambda_i=(x,\nabla u_i(x))$ , $i=0, 1$, be two Lagrangian graph over $\mathbb{T}^n$, which satisfy $[du_0]=[du_1]\in H^1(\mathbb{T}^n)$. Assume the Lagrangian phase satisfy \begin{equation}\label{Critical-Lag-Pha-1}
\vartheta(\Lambda_i)=\sum_{i=1}^n\arctan\lambda_i(\nabla ^2u_i)>\frac{n-1}{2}\pi,~~~i=0,1,
\end{equation}
then there exists $\theta\in (-\pi,\pi]$, such that
\begin{equation}\label{Pos-Lag-Con-1}
  \vartheta(\Lambda_i)\in (\theta-\frac{\pi}{2},\theta+\frac{\pi}{2}),\quad {\rm mod}~2\pi.
\end{equation}
Note that \eqref{Pos-Lag-Con-1} is equivalent to  ${\rm Re}(e^{-\sqrt{-1}\theta}dz_1\wedge\cdots\wedge dz_n)|_{\Lambda_i}>0$, namely, $\Lambda_i\in \mathcal{L}^+(\mathbb{C}^n,\mathbb{T}^n), i=0, 1$, are positive Lagrangians.
In fact, one can take
\begin{equation}\label{Pre-Lag-Phase-1}\theta=\left\{
           \begin{array}{ll}
             0, & \hbox{$n=4l$;} \\
             \frac{1}{2}\pi, & \hbox{$n=4l+1$;} \\
             \pi, & \hbox{$n=4l+2$;} \\
             -\frac{1}{2}\pi, & \hbox{$n=4l+3$,}
           \end{array}
         \right.
\end{equation}
such that \eqref{Pos-Lag-Con-1} holds.

Now we will construct a solution for the Lagrangian geodesic equation,
 \begin{equation}\label{LagGeoEqu_11}
{\rm Im}\Big(e^{-\sqrt{-1}\theta}\det\begin{bmatrix}
 \sqrt{-1}\frac{\partial^2u}{\partial t^2}&\sqrt{-1}\nabla\frac{\partial u}{\partial t}\\ (\sqrt{-1}\nabla\frac{\partial u}{\partial t})^T&I+\sqrt{-1}\nabla^2u
\end{bmatrix}\Big)=0.
\end{equation}
As explained in the introduction section, we try to approximate  the equation \eqref{LagGeoEqu_11} by a family of elliptic equation with a parameter $\tau$,
\begin{equation}\label{LagGeoEquApp_11}
{\rm Im}\Big(e^{-\sqrt{-1}\theta}\det\begin{bmatrix}
 \tau+\sqrt{-1}\frac{\partial^2u}{\partial t^2}&\sqrt{-1}\nabla\frac{\partial u}{\partial t}\\ (\sqrt{-1}\nabla\frac{\partial u}{\partial t})^T&I+\sqrt{-1}\nabla^2u
\end{bmatrix}\Big)=0.
\end{equation}
As $\tau\rightarrow 0$, we recover the Lagranigan geodesic equation \eqref{LagGeoEqu_11} from \eqref{LagGeoEquApp_11}.
Now let us reparametrize the path $u(t,x)$ by scaling the time variable, namely, by introducing the new variable
\[s=\sqrt{\tau}t\in [0,\sqrt{\tau}],\]
then \eqref{LagGeoEquApp_11} can be rewritten as
\begin{equation}\label{LagGeoEquApp_13}
{\rm Im}\Big(e^{-\sqrt{-1}\theta}\det\Big(\begin{bmatrix}
 \sqrt{\tau}&0\\
  0&I
\end{bmatrix}\begin{bmatrix}
  1+\sqrt{-1}\frac{\partial^2}{\partial s^2}u &\sqrt{-1} \nabla\frac{\partial}{\partial s}u\\ (\sqrt{-1} \nabla\frac{\partial}{\partial s}u)^T&I+\sqrt{-1}\nabla^2u
\end{bmatrix}
\begin{bmatrix}
 \sqrt{\tau}&0\\
  0&I
\end{bmatrix}\Big)\Big)=0.
\end{equation}
For $\tau>0$, then \eqref{LagGeoEquApp_13} is equivalent to
\begin{equation}\label{SpeLagGeoEquApp_11}
{\rm Im}\big(e^{-\sqrt{-1}\theta}\det(I+\sqrt{-1}D^2u)\big)=0,
\end{equation}
In particular, the parameter $\tau$ disappear, and we get the special Lagrangian equation, which can be rewritten as a Hessian-type equation:
\begin{equation}\label{SpeLagGeoEquApp_12}
\sum_{i=1}^{n+1}\arctan\lambda_i(D^2u)=k\pi+\theta,\quad {\rm for~some~}~k\in\mathbb{Z}.
\end{equation}
For the $\theta$ defined in \eqref{Pre-Lag-Phase-1}, there exists $k=\frac{n}{2}-\frac{\theta}{\pi}\in \mathbb{Z}$ such that
\begin{equation}\label{SpeLagGeoEqu-Cri-1}
\Theta:=k\pi+\theta=\frac{n}{2}\pi.
\end{equation}
With assumption \eqref{Critical-Lag-Pha-1}, then \[\vartheta(\Lambda_i)>\Theta-\frac{1}{2}\pi, ~~~ i=0, 1,\]
namely, the Lagrangian phase of $u_0$, $u_1$ satisfy  the condition \eqref{Adm-Bou-Con-1} in Theorem \ref{SLE-Cyl-Exi-Thm} automatically.

Consequently, for any $\tau\in (0,1]$, with Theorem \ref{SLE-Cyl-Exi-Thm}, we can construct a smooth solution $u^\tau:=\tilde{u}^\tau+v^{\tau}$ for special Lagrangian equation \eqref{SpeLagGeoEquApp_12}  on the cylinder $[0,\sqrt{\tau}]\times \mathbb{T}^n$. In this case, the right hand of the special Lagrangian equation is  a constant $\Theta=\frac{n}{2}\pi$, then the maximum of first derivative achieved on the boundary. More precisely, from \eqref{Fir-Der-Est-t-1}, i.e. in the case of $\zeta=1$ in the continuous path for the proof of Theorem \ref{SLE-Cyl-Exi-Thm}, we have the uniform estimate for $\hat{v}^\tau(t,x):=v^\tau(\sqrt{\tau}t,x), t\in [0,1],$
\[|\hat{v}^\tau(t,x)|_{C^{1}([0,1]\times \mathbb{T}^n)}\leq C,\]
where $C$ depends only on $u_0$ and $u_1$, not on the parameter $\tau$. For some sequence $\tau_j\rightarrow 0$, if
\[v=\lim_{j\rightarrow \infty} \hat{v}^{\tau_j},\]
then
\[|v|_{C^{1}([0,1]\times \mathbb{T}^n)}\leq C.\]
Finally,  due to the divergence free structure of $\sigma_k$, one can extend the operator
 \begin{equation}\label{SpeLagGeoOpe11}
{\rm Im}\Big(e^{-\sqrt{-1}\theta}\det\begin{bmatrix}
 \sqrt{-1}\frac{\partial^2u}{\partial t^2}&\sqrt{-1}\nabla\frac{\partial u}{\partial t}\\ (\sqrt{-1}\nabla\frac{\partial u}{\partial t})^T&I+\sqrt{-1}\nabla^2u
\end{bmatrix}\Big)
\end{equation}
 to continuous convex functions, and thus we can say $u:=\tilde{u}^1+v$ is a Lipschitz continuous solution of the  Lagrangian geodesic equation \eqref{LagGeoEqu_11}.

Following \cite{TW97}, a function $u\in C^2(\Omega), \Omega \subset \mathbb{R}^n$, is called $k$-convex if $\sigma_k(\nabla^2u)>0$ for $j=1,\cdots, k$. A function $u\in C^0(\Omega)$ is called $k$-convex, if there exists a sequence of $k$-convex functions $u_j\in C^2(\Omega)$ such that in any subdomain
converges uniformly to $u$.
\begin{theorem}\label{Hessian-Measure-1}\cite[Theorem 1.1]{TW97}
For any $k$-convex function $u\in C^0(\Omega)$, there exists a Radon measure $\mu_k[u]$ such that
\begin{enumerate}
  \item if $u\in C^2$, then
  \[\mu_k[u]=\sigma_k(D^2u)dx;\]
  \item if $\{u_j\}\subset C^2$ is a sequence of $k$-convex functions which converges to $u$, then $\mu_k[u_j]\rightarrow \mu_k(u)$ weakly as measure, that is
\[\int_\Omega gd\mu_k[u_j]\rightarrow \int_\Omega gd\mu_k[u],\]
for all $g\in C^0(\Omega)$ with compact support.
\end{enumerate}
\end{theorem}

In our case, since \[\sum_{i=1}^{n+1}\arctan\lambda_i(D^2u^\tau)=\Theta= \frac{n}{2}\pi,\] we conclude that $D^2u^\tau>0$; therefore, $\hat{u}^\tau(t,x):=u^\tau(\sqrt{\tau}t,x)$ is convex in $[0,1]\times \mathbb{T}^n$. Consequently, with Theorem \ref{Hessian-Measure-1}, for any $1\leq k \leq n+1$,  the $k$-Hessian measure
\[\mu_k[\hat{u}^{\tau_j}]\rightarrow \mu_k[\hat{u}] ~~~{\rm in}~~[0,1]\times \mathbb{T}^n;\]
and also
for any $1\leq k \leq n$, then
\[\mu_k[\hat{u}^{\tau_j}(t,-)]\rightarrow \mu_k(\hat{u}(t,-)) ~~~{\rm in}~~\mathbb{T}^n.\]
Note that the operator \eqref{SpeLagGeoOpe11} can be rewritten as a linear combination of $\sigma_k$, namely,
\begin{equation*}
\begin{split}
{\rm Im}\Big(e^{-\sqrt{-1}\theta}&\det\begin{bmatrix}
 \tau+\sqrt{-1}\frac{\partial^2u}{\partial t^2}&\sqrt{-1}\nabla\frac{\partial u}{\partial t}\\ (\sqrt{-1}\nabla\frac{\partial u}{\partial t})^T&I+\sqrt{-1}\nabla^2u
\end{bmatrix}\Big)\\
&={\rm Im}\Big(e^{-\sqrt{-1}\theta}\big(\det(I_{n+1}+\sqrt{-1}D^2_{t,x}u)-(1-\tau)\det(I+\sqrt{-1}\nabla^2u)\big)\Big)\\
&={\rm Im}\Big(e^{-\sqrt{-1}\theta}\big(\sum_{k=0}^{n+1}(\sqrt{-1})^k\sigma_k(D^2_{t,x}u)-(1-\tau)\sum_{k=0}^{n}(\sqrt{-1})^k\sigma_k(\nabla^2u)\big)\Big),
\end{split}
\end{equation*}
thus the operator \eqref{SpeLagGeoOpe11} can be extend to   continuous convex functions;
we can take the limit,
\[u(t,x):=\tilde{u}^1(t,x)+\lim_{\tau_j\rightarrow 0}\hat{v}^{\tau_j}(t,x)\] will be Lipschitz continuous solution of the  Lagrangian geodesic equation \eqref{LagGeoEqu_11} in the weak sense. We finished the proof of Theorem \ref{LGE-Exi-Thm}.

%In this step, we have to establish the continuity of the Hessian operator in the topology of uniform %convergence, then we can say $u$ is a weak Lagrangian geodesic.

\end{document}